\documentclass[11pt]{amsart}

\usepackage{amsmath,amssymb,amscd,amsthm}

\usepackage[dvips]{graphicx}
\usepackage{epsfig}
\usepackage{graphics}
\usepackage{color}
\graphicspath{{Fig/}}

\usepackage{url,hyperref}

\usepackage[matrix,arrow]{xy}

\newtheorem{dfn}{Definition}[section]
\newtheorem{thm}{Theorem}

\newtheorem{lem}[dfn]{Lemma}

\newtheorem{conj}[dfn]{Conjecture}

\theoremstyle{definition}

\def\R{{\mathbb R}}
\def\Z{{\mathbb Z}}

\def\Sph{{\mathbb S}}

\def\phi{\varphi}
\def\epsilon{\varepsilon}

\renewcommand{\mod}{\operatorname{mod}}
\newcommand{\tr}{\operatorname{tr}}
\newcommand{\vol}{\operatorname{vol}}
\newcommand{\dist}{\operatorname{dist}}
\newcommand{\id}{\operatorname{id}}

\title[Discrete Wirtinger inequality and discrete Laplacians]{A general discrete Wirtinger inequality and spectra of discrete Laplacians}
\author{Ivan Izmestiev}
\date{\today}
\thanks{Supported by the European Research Council under the European Union's Seventh Framework Programme (FP7/2007-2013)/\allowbreak ERC Grant agreement no.~247029-SDModels}
\address{Institut f\"ur Mathematik \\
Freie Universit\"at Berlin \\
Arnimallee 2 \\
D-14195 Berlin \\
 GERMANY}
\email{izmestiev@math.fu-berlin.de}

\begin{document}

\begin{abstract}
We prove an inequality that generalizes the Fan-Taussky-Todd discrete analog of the Wirtinger inequality. It is equivalent to an estimate on the spectral gap of a weighted discrete Laplacian on the circle. The proof uses a geometric construction related to the discrete isoperimetric problem on the surface of a cone.

In higher dimensions, the mixed volumes theory leads to similar results, which allows us to associate a discrete Laplace operator to every geodesic triangulation of the sphere and, by analogy, to every triangulated spherical cone-metric. For a cone-metric with positive singular curvatures, we conjecture an estimate on the spectral gap similar to the Lichnerowicz-Obata theorem.
\end{abstract}

\maketitle

\section{Introduction}
\subsection{A general discrete Wirtinger inequality}
The Wirtinger inequality for $2\pi$-periodic functions says
\begin{equation}
\label{eqn:Wirt}
\int_{\Sph^1} f\, dt = 0 \Rightarrow \int_{\Sph^1} (f')^2 \, dt \ge \int_{\Sph^1} f^2 \, dt
\end{equation}
The following elegant theorem from \cite{FTT55} can be viewed as its discrete analog.
\begin{thm}[Fan-Taussky-Todd]
\label{thm:FTT}
For any $x_1, \ldots, x_n \in \R$ such that
\[
\sum_{i=1}^n x_i = 0
\]
the following inequality holds:
\begin{equation}
\label{eqn:FTT}
\sum_{i=1}^n (x_i - x_{i+1})^2 \ge 4 \sin^2 \frac{\pi}n \sum_{i=1}^n x_i^2
\end{equation}
(here $x_{n+1} = x_1$). Equality holds if and only if there exist $a, b \in \R$ such that
\[
x_k = a \cos \frac{2\pi k}n + b \sin \frac{2\pi k}n
\]
\end{thm}
In the same article \cite{FTT55}, similar inequalities for sequences satisfying the boundary conditions $x_0 = 0$ or $x_0 = x_{n+1} = 0$ were proved.
Several different proofs and generalizations followed, \cite{Shi73,MM82,Red83,Che87,Alz91}.

In the present article we prove the following generalization of Theorem \ref{thm:FTT}.

\begin{thm}
\label{thm:WirtDiscr}
For any $x_1, \ldots, x_n \in \R$ and $\alpha_1, \ldots, \alpha_n \in (0, \pi)$ such that
\[
\sum_{i=1}^n \left(\tan\frac{\alpha_i}2 + \tan\frac{\alpha_{i+1}}2\right) x_i = 0, \quad \sum_{i=1}^n \alpha_i \le 2\pi
\]
the following inequality holds:
\begin{equation}
\label{eqn:DiscrWirt}
\sum_{i=1}^n \frac{(x_i - x_{i+1})^2}{\sin\alpha_{i+1}} \ge \sum_{i=1}^n \left( \tan\frac{\alpha_i}2 + \tan\frac{\alpha_{i+1}}2 \right) x_i^2
\end{equation}
Equality holds if and only if $\sum_{i=1}^n \alpha_i = 2\pi$ and there exist $a, b \in \R$ such that
\begin{equation}
\label{eqn:SinCosDiscr}
x_k = a \cos \sum_{i=1}^k \alpha_i + b \sin \sum_{i=1}^k \alpha_i
\end{equation}
If $\sum_{i=1}^n \alpha_i > 2\pi$, then the inequality \eqref{eqn:DiscrWirt} fails for certain values of $x_i$.
% If $\sum_{i=1}^n \alpha_i > 2\pi$, then the inequality \eqref{eqn:DiscrWirt} fails for $x$ as in \eqref{eqn:SinCosDiscr}. \marginpar{indeed?}
\end{thm}

Theorem \ref{thm:FTT} is a special case of Theorem \ref{thm:WirtDiscr} for $\alpha_i = \frac{2\pi}n$.

We obtain Theorem \ref{thm:WirtDiscr} as a consequence of the following.

\begin{thm}
\label{thm:LaplDiscr}
Let $\alpha_1, \ldots, \alpha_n \in (0, \pi)$. Then the circulant tridiagonal $n \times n$ matrix
\begin{equation*}
%\label{eqn:Matrix}
M =
\begin{pmatrix}
-(\cot \alpha_1 + \cot \alpha_2) & \frac1{\sin\alpha_2} & \ldots & \frac1{\sin\alpha_1}\\
\frac1{\sin\alpha_2} & -(\cot \alpha_2 + \cot \alpha_3) & \ddots & \vdots \\
\vdots & \ddots & \ddots & \frac1{\sin\alpha_n} \\
\frac1{\sin\alpha_1} & \ldots & \frac1{\sin\alpha_n} & -(\cot \alpha_n + \cot \alpha_1)
\end{pmatrix}
\end{equation*}
has the signature
\begin{align*}
(2m-1, 2, n-2m-1), &\text{ if } \sum_{i=1}^n \alpha_i = 2m\pi, \, m \ge 1\\
(2m+1, 0, n-2m-1), &\text{ if } 2m\pi < \sum_{i=1}^n \alpha_i < 2(m+1)\pi, \, m \ge 0
\end{align*}
Here $(p,q,r)$ means $p$ positive, $q$ zero, and $r$ negative eigenvalues.

The vector ${\bf 1} = (1, 1, \ldots, 1)$ is always a positive vector for the associated quadratic form:
\[
\langle M {\bf 1}, {\bf 1} \rangle > 0
\]
If $\sum_{i=1}^n \alpha_i \equiv 0 (\mod 2\pi)$, then $\ker M$ consists of all vectors of the form \eqref{eqn:SinCosDiscr}.
% If $\sum_{i=1}^n \alpha_i > 2\pi$, then vectors \eqref{eqn:SinCosDiscr} lie inside the positive cone of $M$. \marginpar{indeed?}
% For any $a, b \in \R$, the vector $x$ as in \eqref{eqn:SinCosDiscr} has the following properties:
% \begin{align*}
% Mx = 0, &\text{ if }\sum_{i=1}^n \alpha_i \equiv 0 (\mod 2\pi)\\
% \langle Mx, x \rangle > 0, &\text{ if } \sum_{i=1}^n \alpha_i > 2\pi
% \end{align*}
\end{thm}

The relation between Theorems \ref{thm:WirtDiscr} and \ref{thm:LaplDiscr} is the same as between the Wirtinger inequality \eqref{eqn:Wirt} and the spectral gap of the Laplacian on $\Sph^1$.
% Namely, the assumption $\int f \, dt = 0$ is the $L^2$-orthogonality to the kernel of $\Delta$, and the conclusion is equivalent to $\langle f, \Delta f \rangle_{L^2} \le \|f\|^2_{L^2}$, that is $\lambda_1 \ge 1$.
Thus we can interpret the matrix $M$ in Theorem \ref{thm:LaplDiscr} as (the weak form) of the operator $\Delta + \id$.
%\marginpar{Laplacian a la Dodziuk.}

\subsection{The discrete isoperimetric problem: a generalization of the L'Huilier theorem}
About two hundred years ago L'Huilier proved that a circumscribed polygon has the greatest area among all polygons with the same side directions and the same perimeter.
Theorems \ref{thm:WirtDiscr} and \ref{thm:LaplDiscr} are related to a certain generalization of the L'Huilier theorem. Again, this imitates the smooth case, as the Wirtinger inequality first appeared in \cite{Bla56} in connection with the isoperimetric problem in the plane.

% Similarly to the smooth case, where the Wirtinger inequality, spectrum of the Laplacian on the circle, and the isoperimetric problem in the plane are closely related,

Define the \emph{euclidean cone of angle} $\omega > 0$ as the space $C_\omega$ resulting from gluing isometrically the sides of an infinite angular region of size $\omega$. (If $\omega > 2\pi$, then paste together several smaller angles, or cut the infinite cyclic branched cover of $\R^2$.)
%For an analytic definition of $C_\omega$ see Section \ref{sec:LHuil}.

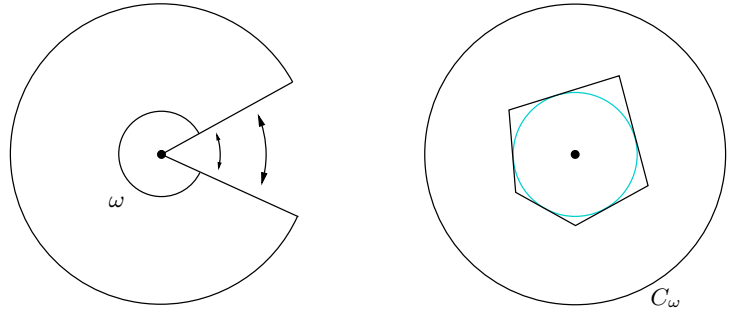
\begin{figure}[ht]
\centering
\begin{picture}(0,0)%
\includegraphics{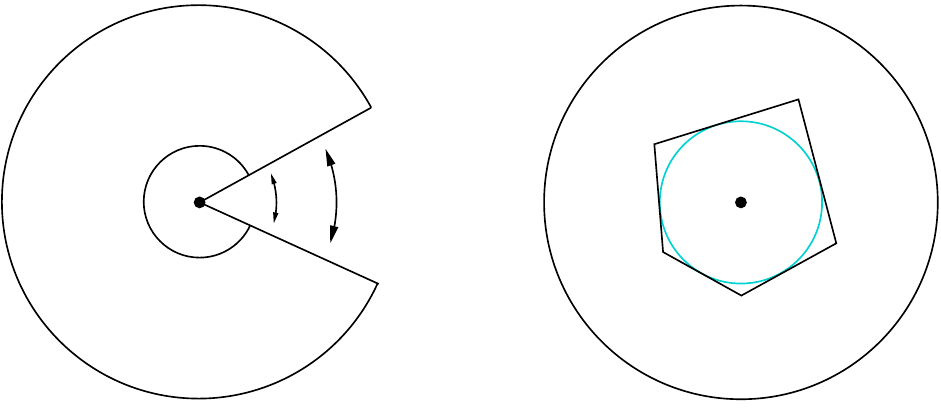}%
\end{picture}%
\setlength{\unitlength}{4144sp}%
\begingroup\makeatletter\ifx\SetFigFont\undefined%
\gdef\SetFigFont#1#2#3#4#5{%
  \reset@font\fontsize{#1}{#2pt}%
  \fontfamily{#3}\fontseries{#4}\fontshape{#5}%
  \selectfont}%
\fi\endgroup%
\begin{picture}(4296,1874)(-12,-1025)
\put(586,-376){\makebox(0,0)[lb]{\smash{{\SetFigFont{9}{10.8}{\rmdefault}{\mddefault}{\updefault}{\color[rgb]{0,0,0}$\omega$}%
}}}}
\put(3826,-961){\makebox(0,0)[lb]{\smash{{\SetFigFont{9}{10.8}{\rmdefault}{\mddefault}{\updefault}{\color[rgb]{0,0,0}$C_\omega$}%
}}}}
\end{picture}%
\caption{The discrete isoperimetric problem on a cone.}
\label{fig:COmega}
\end{figure}

\begin{thm}
\label{thm:IsoperDiscr}
If $\omega \le 2\pi$, then every polygon with the sides tangent to a circle centered at the apex of $C_\omega$ encloses the largest area among all polygons that have the same side directions and the same perimeter.

If $\omega < 2\pi$, then the optimal polygon is unique.

If $\omega = 2\pi$, then the optimal polygon is unique up to translation.

If $\omega > 2\pi$, then the circumcribed polygon is not optimal.
\end{thm}

\subsection{The discrete Wirtinger inequality with Dirichlet boundary conditions}
In a similar way we generalize another inequality from \cite{FTT55}.
\begin{thm}[Fan-Taussky-Todd]
For any $x_1, \ldots, x_n \in \R$ the following inequality holds:
\[
\sum_{i=0}^n (x_i - x_{i+1})^2 \ge 4 \sin^2 \frac{\pi}{2(n+1)} \sum_{i=0}^n x_i^2
\]
where $x_0 = x_{n+1} = 0$. Equality holds if and only if there is $a \in \R$ such that
\[
x_k = a \sin \frac{k\pi}{n+1}
\]
\end{thm}

\begin{thm}
\label{thm:WirtDiscrBdry}
For any $x_0, \ldots, x_{n+1} \in \R$ and $\alpha_1, \ldots, \alpha_{n+1} \in (0, \pi)$ such that
\[
x_0 = x_{n+1} = 0, \quad \sum_{i=1}^n \alpha_i \le \pi
\]
the following inequality holds:
\begin{equation}
\label{eqn:DiscrWirtDir}
\sum_{i=0}^n \frac{(x_i - x_{i+1})^2}{\sin\alpha_{i+1}} \ge \sum_{i=1}^n \left( \tan\frac{\alpha_i}2 + \tan\frac{\alpha_{i+1}}2 \right) x_i^2
\end{equation}
Equality holds if and only if $\sum_{i=1}^{n+1} \alpha_i = \pi$ and there is $a \in \R$ such that
\begin{equation*}
x_k = a \sin \sum_{i=1}^k \alpha_i
\end{equation*}
If $\sum_{i=1}^{n+1} \alpha_i > \pi$, then the inequality \eqref{eqn:DiscrWirt} fails for certain values of $x_i$.
\end{thm}
If $\alpha_i = \frac{\pi}{n+1}$ for all $i$, then this becomes a Fan-Taussky-Todd inequality.

Similarly to the above, the inequality follows from a theorem about the signature of a tridiagonal (this time non-circulant) matrix, see Section \ref{sec:WirtBdry}. It is related to a discrete version of the Dido isoperimetric problem.

\subsection{Related work}
Milovanovi{\'c} and Milovanovi{\'c} \cite{MM82} studied the question of finding optimal constants $A$ and $B$ in the inequalities
\[
A \sum_{i=0}^n p_i x_i^2 \le \sum_{i=0}^n r_i(x_i-x_{i+1})^2 \le B \sum_{i=0}^n p_i x_i^2
\]
for given sequences $(p_i)$ and $(r_i)$. They dealt only with the Dirichlet boundary conditions $x_0 = 0$ or $x_0 = x_{n+1} = 0$, and the answer is rather implicit: $A$ and $B$ are the minimum and the maximum zeros of a recursively defined polynomial (the characteristic polynomial of the corresponding quadratic form).

There is a partial generalization of Theorems \ref{thm:WirtDiscr} and \ref{thm:LaplDiscr} to higher dimensions. Instead of the angles $\alpha_1, \ldots, \alpha_n$, one fixes a geodesic Delaunay triangulation of $\Sph^{d-1}$, and the matrix $M$ is defined as the Hessian of the volume of polytopes whose normal fan is the given triangulation. The signature of $M$ follows from the Minkowski inequality for mixed volumes. A full generalization would deal with a Delaunay triangulated spherical cone-metric on $\Sph^{d-1}$ with positive singular curvatures, and would be a discrete analog of the Lichnerowicz theorem on the spectral gap for metrics with Ricci curvature bounded below. See \cite{Izm14} and Section \ref{sec:HighDim} below for details.

The spectral gap of the Laplacian on ``short circles'' plays a crucial role in the rigidity theorems for hyperbolic cone-manifolds with positive singular curvatures \cite{HK98,MM09,Wei09} based on Cheeger's extension of the Hodge theory to singular spaces \cite{Che83}. As elementary as it is, Theorem \ref{thm:WirtDiscr} could provide a basis for spectral estimates for natural discrete Laplacians, and in particular an alternative approach to the rigidity of cone-manifolds.

\subsection{Acknowledgment}
This article was written during author's visit to the Pennsylvania State University.

\section{Wirtinger, Laplace, and isoperimetry in the smooth case}
\label{sec:SmoothCase}
\subsection{Wirtinger's inequality and the spectral gap}
\begin{thm}[Wirtinger's lemma]
\label{thm:WirtSm}
Let $f \colon \Sph^1 \to \R$ be a $C^\infty$-function with zero average:
\[
\int_{\Sph^1} f(t) \, dt = 0
\]
Then
\[
\int_{\Sph^1} f^2(t) \, dt \le \int_{\Sph^1} (f')^2 \, dt
\]
Equality holds if and only if
\begin{equation}
\label{eqn:SinCos}
f(t) = a \cos t + b \sin t
\end{equation}
for some $a, b \in \R$.
\end{thm}

\begin{thm}[Spectrum of the Laplacian]
\label{thm:LaplSm}
The spectrum of the Laplace operator
\[
\Delta f = f'' \text{ for } f \in C^\infty(\Sph^1)
\]
is $\{-k^2 \mid k \in \Z\}$. The zero eigenspace consists of the constant functions; the eigenvalue $-1$ is double, and the associated eigenspace consists of the functions of the form \eqref{eqn:SinCos}.
\end{thm}

Theorem \ref{thm:WirtSm} is equivalent to the fact that the spectral gap of the Laplace operator equals $1$. Indeed, the zero average condition can be rewritten as
\[
\langle f, 1 \rangle_{L^2} = 0
\]
that is $f$ is $L^2$-orthogonal to the kernel of the Laplacian. This implies
\[
\int_{\Sph^1} (f')^2 \, dt = - \int_{\Sph^1} f'' \cdot f \, dt = - \langle \Delta f, f \rangle_{L^2} \ge \lambda_1 \|f\|^2
\]
which is the Wirtinger inequality since $\lambda_1 = 1$. Equality holds only for the eigenfunctions of $\lambda_1$.

\subsection{Wirtinger's inequality and the isoperimetric problem}
Blaschke used Wirtinger's inequality in 1916 to prove Minkowski's inequality in the plane, and by means of it the isoperimetric inequality \cite[\S 23]{Bla56}. For historic references, see \cite{Mit70}.

\begin{thm}[Isoperimetric problem in the plane]
\label{thm:IsoperSm}
Among all convex closed $C^2$-curves in the plane with the total length $2\pi$, the unit circle encloses the largest area.
\end{thm}

Below is Blaschke-Wirtinger's argument, with a shortcut avoiding the more general Minkowski inequality.

Let $\Gamma$ be a convex closed curve in $\R^2$. Define the \emph{support function} of $\Gamma$ as
\[
h \colon \Sph^1 \to \R, \quad h(t) = \max\{\langle x, t \rangle \mid x \in \Gamma\}
\]
(Here $\Sph^1$ is viewed as the set of unit vectors in $\R^2$.)
If $\Gamma$ is strictly convex and of class $C^2$, then the Gauss map $\Gamma \to \Sph^1$ is a diffeomorphism. The corresponding parametrization $\gamma \colon \Sph^1 \to \Gamma$ of $\Gamma$ by its normal has the form
\[
\gamma(t) = h t + \nabla h
\]
The perimeter of $\Gamma$ and the area of the enclosed region can be computed as
\[
L(\Gamma) = \int_{\Sph^1} h \, dt, \quad
A(\Gamma) = \frac12 \int_{\Sph^1} h (h + h'')\, dt = \frac12 \int_{\Sph^1} (h^2 - (h')^2) \, dt
\]

Now assume $L(\Gamma) = 2\pi$ and put $f(t) = h(t) - 1$. We have
\[
\int_{\Sph^1} f(t) \, dt = L(\Gamma) - 2\pi = 0
\]
% Hence by Wirtinger's lemma we have $\int f^2 \le \int (f')^2$.
It follows that
\[
\int_{\Sph^1} h^2(t) \, dt = \int_{\Sph^1} (1 + f(t))^2 \, dt = 2\pi + \int_{\Sph^1} f^2(t) \, dt
\]
Hence
\[
A(\Gamma) = \frac12 \int_{\Sph^1} (h^2(t) - (h'(t))^2) \, dt = 2\pi + \frac12 \int_{\Sph^1} (f^2(t) - (f'(t))^2) \, dt \ge 2\pi
\]
by the Wirtinger inequality.

It is also possible to derive Wirtinger's inequality from the isoperimetric one: start with a twice differentiable function $f$ and choose $\epsilon > 0$ small enough so that $1 + \epsilon f$ is the support function of a convex curve.

See \cite{Scn14} for the general theory of convex bodies, and \cite{Trei09} for a nice survey on the isoperimetry and Minkowski theory.

\section{Wirtinger, Laplace, and isoperimetry in the discrete case}
Since we will use geometric objects in our proof of Theorem \ref{thm:LaplDiscr}, let us start with geometry.
\subsection{The geometric setup}
\label{sec:LHuil}
% The cone $C_\omega$ is the space
% \[
% [0, +\infty) \times (\R/\omega \Z) / (0,\phi_1) \sim (0,\phi_2)
% \]
% with a geodesic path metric coming from the Riemannian metric $d\rho^2 + \rho^2 d\phi^2$ on $C_\omega \setminus \{0\}$. By scaling the coordinate $\phi$ one can rewrite this as the metric
% \[
% d\rho^2 + \left(\frac{\omega}{2\pi}\right)^2 d\phi^2
% \]
% in the usual polar coordinates on $\R^2$.

Take $n$ infinite angular regions $A_1, \ldots, A_n$ of angles $\alpha_1, \ldots, \alpha_n \in (0,\pi)$ respectively and glue them along their sides in this cyclic order. This results in a cone $C_\omega$ with $\omega = \sum_{i=1}^n \alpha_i$. Let $R_i$ be the ray separating $A_i$ from $A_{i+1}$, and let $\nu_i$ be the unit vector along $R_i$ pointing away from the apex. See Figure \ref{fig:ConeAngles}, left.

\begin{figure}[ht]
\centering
\begin{picture}(0,0)%
\includegraphics{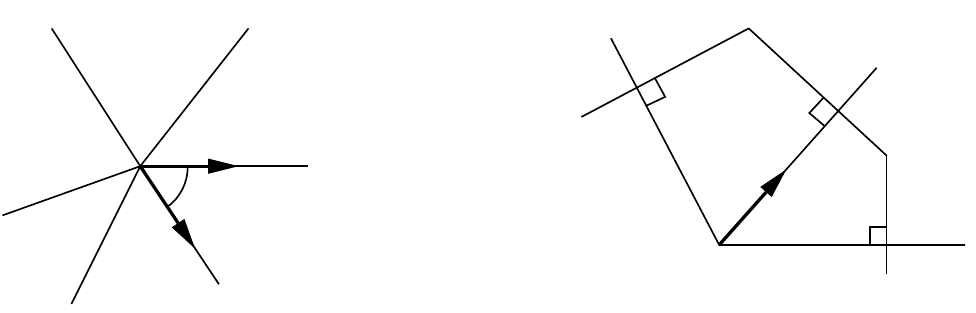}%
\end{picture}%
\setlength{\unitlength}{4144sp}%
\begingroup\makeatletter\ifx\SetFigFont\undefined%
\gdef\SetFigFont#1#2#3#4#5{%
  \reset@font\fontsize{#1}{#2pt}%
  \fontfamily{#3}\fontseries{#4}\fontshape{#5}%
  \selectfont}%
\fi\endgroup%
\begin{picture}(4450,1421)(259,-955)
\put(1276,-895){\makebox(0,0)[lb]{\smash{{\SetFigFont{9}{10.8}{\rmdefault}{\mddefault}{\updefault}{\color[rgb]{0,0,0}$R_n$}%
}}}}
\put(1410,310){\makebox(0,0)[lb]{\smash{{\SetFigFont{9}{10.8}{\rmdefault}{\mddefault}{\updefault}{\color[rgb]{0,0,0}$R_2$}%
}}}}
\put(1110,-470){\makebox(0,0)[lb]{\smash{{\SetFigFont{9}{10.8}{\rmdefault}{\mddefault}{\updefault}{\color[rgb]{0,0,0}$\alpha_1$}%
}}}}
\put(1712,-343){\makebox(0,0)[lb]{\smash{{\SetFigFont{9}{10.8}{\rmdefault}{\mddefault}{\updefault}{\color[rgb]{0,0,0}$R_1$}%
}}}}
\put(1014,-726){\makebox(0,0)[lb]{\smash{{\SetFigFont{9}{10.8}{\rmdefault}{\mddefault}{\updefault}{\color[rgb]{0,0,0}$\nu_n$}%
}}}}
\put(1289,-252){\makebox(0,0)[lb]{\smash{{\SetFigFont{9}{10.8}{\rmdefault}{\mddefault}{\updefault}{\color[rgb]{0,0,0}$\nu_1$}%
}}}}
\put(1454,-614){\makebox(0,0)[lb]{\smash{{\SetFigFont{9}{10.8}{\rmdefault}{\mddefault}{\updefault}{\color[rgb]{0,0,0}$A_1$}%
}}}}
\put(1512, 35){\makebox(0,0)[lb]{\smash{{\SetFigFont{9}{10.8}{\rmdefault}{\mddefault}{\updefault}{\color[rgb]{0,0,0}$A_2$}%
}}}}
\put(4694,-702){\makebox(0,0)[lb]{\smash{{\SetFigFont{9}{10.8}{\rmdefault}{\mddefault}{\updefault}{\color[rgb]{0,0,0}$R_{i-1}$}%
}}}}
\put(4292,138){\makebox(0,0)[lb]{\smash{{\SetFigFont{9}{10.8}{\rmdefault}{\mddefault}{\updefault}{\color[rgb]{0,0,0}$R_i$}%
}}}}
\put(4221,-880){\makebox(0,0)[lb]{\smash{{\SetFigFont{9}{10.8}{\rmdefault}{\mddefault}{\updefault}{\color[rgb]{0,0,0}$L_{i-1}$}%
}}}}
\put(2691,-205){\makebox(0,0)[lb]{\smash{{\SetFigFont{9}{10.8}{\rmdefault}{\mddefault}{\updefault}{\color[rgb]{0,0,0}$L_{i+1}$}%
}}}}
\put(3878,-411){\makebox(0,0)[lb]{\smash{{\SetFigFont{9}{10.8}{\rmdefault}{\mddefault}{\updefault}{\color[rgb]{0,0,0}$\nu_i$}%
}}}}
\put(3850,-619){\makebox(0,0)[lb]{\smash{{\SetFigFont{9}{10.8}{\rmdefault}{\mddefault}{\updefault}{\color[rgb]{0,0,0}$x_{i-1}$}%
}}}}
\put(3229,-347){\makebox(0,0)[lb]{\smash{{\SetFigFont{9}{10.8}{\rmdefault}{\mddefault}{\updefault}{\color[rgb]{0,0,0}$x_i$}%
}}}}
\put(2840,319){\makebox(0,0)[lb]{\smash{{\SetFigFont{9}{10.8}{\rmdefault}{\mddefault}{\updefault}{\color[rgb]{0,0,0}$R_{i+1}$}%
}}}}
\put(3903,143){\makebox(0,0)[lb]{\smash{{\SetFigFont{9}{10.8}{\rmdefault}{\mddefault}{\updefault}{\color[rgb]{0,0,0}$\ell_i$}%
}}}}
\end{picture}%
\caption{The geometric setup for the isoperimetric problem on the cone.}
\label{fig:ConeAngles}
\end{figure}
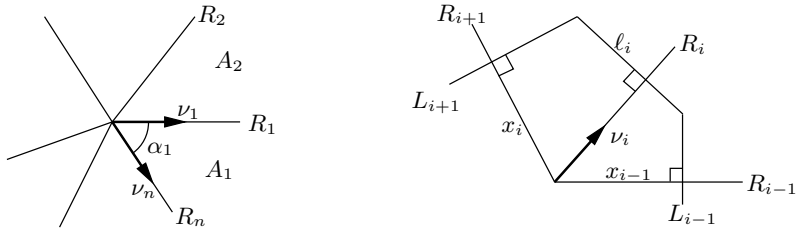

Develop the angle $A_i \cup A_{i+1}$ into the plane, choose $x_{i-1}, x_i, x_{i+1} \in \R$ and draw the lines
\[
L_j = \{p \in \R^2 \mid \langle p, \nu_i \rangle = x_j\}, j = i-1, i, i+1
\]
Orient the line $L_i$ as pointing from $A_i$ into $A_{i+1}$ and denote by $\ell_i$ the signed length of the segment with the endpoints $L_i \cap L_{i-1}$ and $L_i \cap L_{i+1}$. A simple computation yields
\begin{equation}
\label{eqn:EllI}
\ell_i = \frac{x_{i-1} - x_i\cos\alpha_i}{\sin\alpha_i} + \frac{x_{i+1} - x_i\cos\alpha_{i+1}}{\sin\alpha_{i+1}}
\end{equation}
This defines a linear operator $\ell \colon \R^n \to \R^n$. It turns out that $\ell(x) = Mx$, where $M$ is the matrix from Theorem \ref{thm:LaplDiscr}.

\subsection{Proof of the signature theorem}
\label{sec:SignProof}
\begin{lem}
\label{lem:Rank}
The corank of the matrix $M$ from Theorem \ref{thm:LaplDiscr} is as follows.
\[
\dim \ker M =
\begin{cases}
0, &\text{if } \sum_{i=1}^n \alpha_i \equiv 0 (\mod 2\pi)\\
2, &\text{if } \sum_{i=1}^n \alpha_i \not\equiv 0 (\mod 2\pi)
\end{cases}
\]
\end{lem}
\begin{proof}
We will show that the elements of $\ker M$ are in a one-to-one correspondence with parallel $1$-forms on $C_{\omega} \setminus \{0\}$. If $\omega \not\equiv 0(\mod 2\pi)$, then every parallel form vanishes. If $\omega \equiv 0(\mod 2\pi)$, then all of them are pullbacks of parallel forms on $\R^2$ via the developing map, and thus $\ker M$ has dimension~$2$. This will imply the statement of the lemma.

With any $x = (x_1, \ldots, x_n) \in \R^n$ associate a family of $1$-forms $\xi_i \in \Omega^1(A_i)$ where each $\xi_i$ is parallel on $A_i$ and is determined by
\[
\xi_i(\nu_{i-1}) = x_{i-1}, \quad \xi_i(\nu_i) = x_i
\]
Here $\nu_i$ denotes, by abuse of notation, the extension of the vector $\nu_i$ to a parallel vector field on $A_i \cup A_{i+1}$.
We claim that $x \in \ker M$ if and only if the form $\xi_i$ is parallel to $\xi_{i+1}$ for all $i$.

\begin{figure}[ht]
\centering
\begin{picture}(0,0)%
\includegraphics{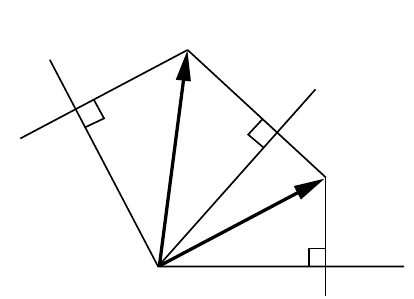}%
\end{picture}%
\setlength{\unitlength}{4144sp}%
\begingroup\makeatletter\ifx\SetFigFont\undefined%
\gdef\SetFigFont#1#2#3#4#5{%
  \reset@font\fontsize{#1}{#2pt}%
  \fontfamily{#3}\fontseries{#4}\fontshape{#5}%
  \selectfont}%
\fi\endgroup%
\begin{picture}(1884,1354)(2825,-793)
\put(4292,138){\makebox(0,0)[lb]{\smash{{\SetFigFont{9}{10.8}{\rmdefault}{\mddefault}{\updefault}{\color[rgb]{0,0,0}$R_i$}%
}}}}
\put(3229,-347){\makebox(0,0)[lb]{\smash{{\SetFigFont{9}{10.8}{\rmdefault}{\mddefault}{\updefault}{\color[rgb]{0,0,0}$x_i$}%
}}}}
\put(2840,319){\makebox(0,0)[lb]{\smash{{\SetFigFont{9}{10.8}{\rmdefault}{\mddefault}{\updefault}{\color[rgb]{0,0,0}$R_{i+1}$}%
}}}}
\put(3903,143){\makebox(0,0)[lb]{\smash{{\SetFigFont{9}{10.8}{\rmdefault}{\mddefault}{\updefault}{\color[rgb]{0,0,0}$\ell_i$}%
}}}}
\put(4694,-702){\makebox(0,0)[lb]{\smash{{\SetFigFont{9}{10.8}{\rmdefault}{\mddefault}{\updefault}{\color[rgb]{0,0,0}$R_{i-1}$}%
}}}}
\put(3850,-619){\makebox(0,0)[lb]{\smash{{\SetFigFont{9}{10.8}{\rmdefault}{\mddefault}{\updefault}{\color[rgb]{0,0,0}$x_{i-1}$}%
}}}}
\put(4375,-272){\makebox(0,0)[lb]{\smash{{\SetFigFont{9}{10.8}{\rmdefault}{\mddefault}{\updefault}{\color[rgb]{0,0,0}$X_i$}%
}}}}
\put(3635,414){\makebox(0,0)[lb]{\smash{{\SetFigFont{9}{10.8}{\rmdefault}{\mddefault}{\updefault}{\color[rgb]{0,0,0}$X_{i+1}$}%
}}}}
\end{picture}%
\caption{Vectors $X_i$ and $X_{i+1}$ dual to the forms $\xi_i$ and $\xi_{i+1}$.}
\label{fig:CompareForms}
\end{figure}
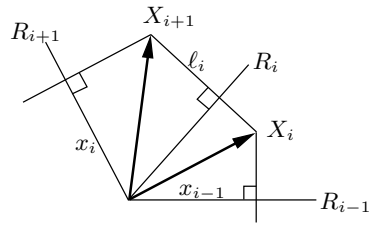

To compare the forms $\xi_i$ and $\xi_{i+1}$, develop the angle $A_i \cup A_{i+1}$ on the plane. We have
\[
\xi_i(v) = \langle X_i, v \rangle,
\]
where $X_i \in \R^2$ is the vector whose projections to the rays $R_{i-1}$ and $R_i$ have lengths $x_{i-1}$ and $x_i$, respectively, see Figure \ref{fig:CompareForms}. Thus $\xi_i$ is parallel to $\xi_{i+1}$ if and only if $X_i = X_{i+1}$. On the other hand, by Section \ref{sec:LHuil} we have
\[
\|X_{i+1} - X_i\| = |\ell_i(x)|
\]
Hence $1$-forms $\xi_i$ define a parallel form on $C_{\omega} \setminus \{0\}$ if and only if $Mx = 0$.
\end{proof}

\begin{proof}[Proof of Theorem \ref{thm:LaplDiscr}]
Put $\omega = \sum_{i=1}^n \alpha_i$ and define
\[
\alpha_i(t) = (1-t)\alpha_i + t \frac{\omega}n, \quad t \in [0,1]
\]
For all $t$ we have $\alpha_i(t) \in (0,\pi)$ and $\sum_{i=1}^n \alpha_i(t) = \omega$. Hence, by Lemma \ref{lem:Rank} the matrix $M_t$ constructed from the angles $\alpha_i(t)$ has a constant rank for all~$t$. Therefore its signature does not depend on $t$. It remains to determine the signature of the matrix $M_1$. After scaling by a positive factor $M_1$ becomes
\[
\begin{pmatrix}
-2\cos\frac{\omega}n & 1 & \ldots & 1\\
1 & -2\cos\frac{\omega}n & \ddots & \vdots \\
\vdots & \ddots & \ddots & 1 \\
1 & \ldots & 1 & -2\cos\frac{\omega}n
\end{pmatrix}
\]
The eigenvalues of this matrix are
\[
\left\{\left.2\cos \frac{2\pi k}{n} - 2 \cos \frac{\omega}n \right| k = 1, 2, \ldots, n\right\}
\]
If $\omega = 2\pi m$, then exactly two of these eigenvalues are zero (the ones with $k = m$ and $k = n-m$). For $\omega > 2\pi m$ there are exactly $2m + 1$ positive eigenvalues. The theorem is proved.
\end{proof}

\subsection{Proof of the general discrete Wirtinger inequality}
\label{sec:GenWirtProof}
Let us show that Theorem \ref{thm:LaplDiscr} implies Theorem \ref{thm:WirtDiscr}.

The key point is that inequality \eqref{eqn:DiscrWirt} is equivalent to $\langle Mx, x \rangle \le 0$ and that
\[
\sum_{i=1}^n \left(\tan\frac{\alpha_i}2 + \tan\frac{\alpha_{i+1}}2\right) x_i = \langle Mx, {\bf 1} \rangle
\]
Assume first $\sum_{i=1}^n \alpha_i \le 2\pi$. By Theorem \ref{thm:LaplDiscr}, the quadratic form $M$ has positive index $1$ and takes a positive value on the vector ${\bf 1}$. Hence it is negative semidefinite on the orthogonal complement to ${\bf 1}$:
\[
\langle Mx, {\bf 1} \rangle = 1 \Rightarrow \langle Mx, x \rangle \le 0
\]
This proves the first statement of Theorem \ref{thm:WirtDiscr}.

If $\sum_{i=1}^n \alpha_i < 2\pi$, then $M$ is negative definite on the complement to ${\bf 1}$, hence equality holds in \eqref{eqn:DiscrWirt} only for $x = 0$. If $\sum_{i=1}^n \alpha_i = 2\pi$, then equality holds only if $Mx = 0$ (all isotropic vectors of a semidefinite quadratic form lie in its kernel). We have $Mx = 0$ if and only if all vectors $X_i$ on Figure \ref{fig:CompareForms} are equal, that is iff $x_i = \langle X, \nu_i \rangle$ for some $X \in \R^2$.
This proves the second statement of Theorem \ref{thm:WirtDiscr}.

Finally, under the assumption $\sum_{i=1}^n \alpha_i > 2\pi$ the quadratic form $\langle Mx, x \rangle$ is indefinite on the orthogonal complement to ${\bf 1}$, hence the inequality \eqref{eqn:DiscrWirt} fails for some $x$.

% Finally, the fact that under assumption $\sum_{i=1}^n \alpha_i > 2\pi$ we have $\langle Mx, x \rangle > 0$ for $x$ as in \eqref{eqn:SinCosDiscr} implies the third part of Theorem \ref{thm:WirtDiscr}.

\subsection{Proof of the isoperimetric inequality}
First we have to define a convex polygon on $C_\omega$ with given side directions. Let $C_\omega$ be assembled from the angular regions $A_i$ as in Section \ref{sec:LHuil} and let $x_1, \ldots, x_n > 0$. Then we can draw the lines $L_i$ as described in Section \ref{sec:LHuil} directly on $C_\omega$. If $\omega < 2\pi$, then $L_{i-1}$ and $L_i$ may intersect in more than one point, but their lifts to the universal branched cover have only one point in common. Denote the projection of this point to $C_\omega$ by $p_i$. We obtain a closed polygonal line $p_1 \ldots p_n$ with sides lying on $L_i$. If $\ell_i(x) > 0$, then we call this line a \emph{convex polygon} on $C_\omega$ with the exterior normals $\nu_1, \ldots, \nu_n$ and \emph{support numbers} $x_1, \ldots, x_n$.

The polygon with the support numbers ${\bf 1}$ is circumscribed about the unit circle centered at the apex.

\begin{proof}[Proof of Theorem \ref{thm:IsoperDiscr}]
The perimeter and the area of a convex polygon with the support numbers $h$ are computed as follows.
\begin{gather*}
L(h) = \sum_{i=1}^n \ell_i(h) = \langle Mh, 1 \rangle\\
A(h) = \frac12 \sum_{i=1}^n h_i \ell_i(h) = \frac12 \langle Mh, h \rangle
\end{gather*}
It suffices to prove the theorem in the special case of a polygon circumscribed about the unit circle, that is we need to show
\[
L(h) = L({\bf 1}) \Rightarrow A(h) \le A({\bf 1})
\]
Put $f = h - {\bf 1} \in \R^n$. Due to the assumption $L(h) = L({\bf 1})$ we have
\[
\langle Mf, {\bf 1} \rangle = 0
\]
Hence by Theorem \ref{thm:LaplDiscr} we have $\langle Mf, f \rangle \le 0$, so that
\begin{multline*}
A(h) = \frac12 \langle M({\bf 1} + f), {\bf 1} + f \rangle = \frac12 \langle M {\bf 1}, {\bf 1} \rangle + \langle Mf, {\bf 1} \rangle + \frac12 \langle Mf, f \rangle\\
= A({\bf 1}) + \frac12 \langle Mf, f \rangle \le A({\bf 1})
\end{multline*}
The statements on the uniqueness and optimality follow from the facts about the signature of $M$ and the values of $M$ on the vectors \eqref{eqn:SinCosDiscr}.
\end{proof}

\section{The Wirtinger inequality with boundary conditions}
\label{sec:WirtBdry}
For functions vanishing at the endpoints of an interval we have the following.
\begin{thm}
Let $f \colon [0,\pi] \to \R$ be a $C^\infty$-function such that $f(0) = f(\pi) = 0$.
Then
\[
\int_0^\pi f^2(t) \, dt \le \int_0^\pi (f')^2(t) \, dt
\]
Equality holds if and only if $x = a \sin t$.
\end{thm}
There is an obvious relation to the Dirichlet spectrum of the Laplacian.

In a way similar to this and to the argument in Section \ref{sec:GenWirtProof}, Theorem \ref{thm:WirtDiscrBdry} is implied by the following.

\begin{thm}
\label{thm:LaplDiscrBdry}
Let $\alpha_1, \ldots, \alpha_{n+1} \in (0, \pi)$. Then the tridiagonal $n \times n$ matrix
\begin{equation*}
%\label{eqn:Matrix}
M =
\begin{pmatrix}
-(\cot \alpha_1 + \cot \alpha_2) & \frac1{\sin\alpha_2} & \ldots & 0\\
\frac1{\sin\alpha_2} & -(\cot \alpha_2 + \cot \alpha_3) & \ddots & \vdots \\
\vdots & \ddots & \ddots & \frac1{\sin\alpha_n} \\
0 & \ldots & \frac1{\sin\alpha_n} & -(\cot \alpha_n + \cot \alpha_{n+1})
\end{pmatrix}
\end{equation*}
has the signature
\begin{align*}
(m-1, 1, n-m), &\text{ if } \sum_{i=1}^{n+1} \alpha_i = m\pi, \, m \ge 1\\
(m, 0, n-m), &\text{ if } m\pi < \sum_{i=1}^{n+1} \alpha_i < (m+1)\pi, \, m \ge 0
\end{align*}
Here $(p,q,r)$ means $p$ positive, $q$ zero, and $r$ negative eigenvalues.

If $\sum_{i=1}^{n+1} \alpha_i = m\pi$, then $\ker M$ consists of the vectors of the form
\[
x_k = a \sin \sum_{i=1}^k \alpha_i
\]
\end{thm}
\begin{proof}
Similarly to Section \ref{sec:SignProof}, consider the angular region $A_\omega$ glued out of $n$ regions $A_i$ of the angles $\alpha_i$.

First show that $\dim \ker M = 1$ if $\sum_{i=1}^{n+1} \alpha_i = m\pi$ and $\dim \ker M = 0$ otherwise. For this, associate as in Section \ref{sec:SignProof} with every element of the kernel a parallel $1$-form $\xi$ on $A_\omega$ such that $\xi(\nu_0) = \xi(\nu_{n+1}) = 0$. Since the angle between $\nu_0$ and $\nu_{n+1}$ is $\omega$, such a form exists only if $\omega = m\pi$.

Then deform the angles $\alpha_i$, while keeping their sum fixed, to $\alpha_i = \frac{\omega}{n+1}$ and use the fact that the matrix
\[
\begin{pmatrix}
-2\cos\frac{\omega}{n+1} & 1 & \ldots & 0\\
1 & -2\cos\frac{\omega}{n+1} & \ddots & \vdots \\
\vdots & \ddots & \ddots & 1 \\
0 & \ldots & 1 & -2\cos\frac{\omega}{n+1}
\end{pmatrix}
\]
has the spectrum
\[
\left\{\left.2\cos \frac{\pi k}{n+1} - 2 \cos \frac{\omega}{n+1} \right| k = 1, 2, \ldots, n\right\}
\]
It follows that the signature of $M$ is as stated in the theorem.
\end{proof}

% Fan, Taussky, and Todd proved \cite{FTT55} the following discrete analog.
% \begin{thm}
% For any $x_0, x_1, \ldots, x_n, x_{n+1} \in \R$ such that $x_0 = x_{n+1} = 0$ the following inequality holds:
% \[
% \sum_{i=0}^n (x_i - x_{i+1})^2 \ge 4 \sin^2 \frac{\pi}{2(n+1)} \sum_{i=1}^n x_n^2
% \]
% Equality holds if and only if
% \[
% x_k = a \sin \frac{k\pi}{n+1}, \quad k = 1, \ldots, n
% \]
% \end{thm}

\section{Higher dimensions}
\label{sec:HighDim}
\subsection{The quermassintegrals}
\label{sec:Quermass}
\begin{dfn}
The $i$-th \emph{quermassintegral} $W_i(K)$ of a convex body $K \subset \R^n$ is the coefficient in the expansion
\[
\vol_n(K_t) = \sum_{i=0}^n t^i \binom{n}{i} W_i(K)
\]
where $K_t = \{x \in \R^n \mid \dist(x,K) \le t\}$ is the $t$-neighborhood of $K$.
\end{dfn}
In particular,
\[
W_0(K) = \vol_n(K), \quad W_1(K) = \frac1n \vol_{n-1}{\partial K}, \quad W_n(K) = \vol_n(B^n)
\]
Also, $W_i$ is proportional to the mean volume of the projections of $K$ to $(n-i)$-dimensional subspaces, as well as to the integral of the $(i-1)$-st homogeneous polynomial in the principal curvatures (provided $\partial K$ is smooth):
\[
W_i(K) = c_{n,i} \int_{{\mathrm Gr}(n,n-i)} \vol_{n-i}({\mathrm pr}_\xi(K)) \, d\xi =
c'_{n,i} \int_{\partial K} \sigma_{i-1} \, dx
\]
We will need the following expressions for $W_{n-1}$ and $W_{n-2}$ in terms of the support function $h \colon \Sph^{n-1} \to \R$:
\[
W_{n-1}(K) = \frac1n \int_{\Sph^{n-1}} h \, d\nu, \quad W_{n-2}(K) = \frac1n \int_{\Sph^{n-1}} h((n-1)h + \Delta h) \, d\nu
\]
See e.g. \cite{Scn14} for a proof.

\subsection{The smooth case}
The following three theorems generalize those from Section \ref{sec:SmoothCase}. Again, Theorem \ref{thm:LaplSmHigh} implies the other two, where for Theorem \ref{thm:AvWidth} one needs the formulas for $W_{n-1}$ and $W_{n-2}$ from the previous section.
\begin{thm}
\label{thm:WirtSmHigh}
Let $f \colon \Sph^{n-1} \to \R$ be a $C^1$-function with the zero average:
\[
\int_{\Sph^{n-1}} f(x) \, dx = 0
\]
Then
\[
\int_{\Sph^{n-1}} f^2(x) \, dx \le \frac1{n-1} \int_{\Sph^{n-1}} \|\nabla f\|^2 \, dx
\]
Equality holds if and only if $f$ is a spherical harmonic of order $1$, that is a restriction to $\Sph^{n-1}$ of a linear function on $\R^n$.
\end{thm}

\begin{thm}
\label{thm:LaplSmHigh}
The spectrum of the Laplace-Beltrami operator
\[
\Delta f = \tr \nabla^2 f
\]
on the unit sphere $\Sph^{n-1}$ is $\{-k(k+n-2) \mid k \in \Z\}$. The zero eigenspace consists of the constant functions; the eigenvalue $-(n-1)$ has multiplicity $n$, and the associated eigenspace consists of the restrictions of linear functions on $\R^n$.
\end{thm}

\begin{thm}
\label{thm:AvWidth}
Among all convex bodies in $\R^n$ with smooth boundary and with the average width $2$, the unit ball has the largest average projection area to the $2$-dimensional subspaces.
\end{thm}

\subsection{The discrete Laplacian and the discrete Lichnerowicz conjecture}
The quermassintegrals are defined for all convex bodies, and in particular for convex polyhedra. For a convex polyhedron $P(h)$ with fixed outward unit facet normals $\nu_1, \ldots, \nu_n$ and varying support numbers $h_1, \ldots, h_n$, the quermassintegral $W_{n-2}(h)$ is a quadratic form in $h$, provided that the combinatorial type of $P(h)$ does not change.

\begin{dfn}
Let $\nu_1, \ldots, \nu_n$ be in general position, so that all polyhedra $P(h)$ with $h$ close to ${\bf 1}$ have the same combinatorics.
Denote by $M = M(\nu)$ the symmetric $n \times n$-matrix such that
\[
W_{n-2}(h) = \langle Mh, h \rangle
\]
\end{dfn}
Due to the last formula from Section \ref{sec:Quermass}, the self-adjoint operator $M$ is the discrete analog of the operator $(n-1)\id + \Delta$.

\begin{thm}
\label{thm:WirtDiscrHigh}
If $x \in \R^n$ is such that $\langle Mx, {\bf 1} \rangle = 0$, then $\langle Mx, x \rangle \le 0$. Equality holds if and only if there is $X \in \R^n$ such that $x_i = \langle X, \nu_i \rangle$ for all $i$.
\end{thm}

It is possible to define the matrix $M$ for any triangulation of $\Sph^{n-1}$ whose simplices are equipped with a spherical metric.

\begin{conj}
If all cone angles in a spherical cone metric on $\Sph^{n-1}$ are less that $2\pi$, and the triangulation is Delaunay, then
\[
\langle Mx, {\bf 1} \rangle = 0 \Rightarrow \langle Mx, x \rangle \le 0
\]
\end{conj}

This conjecture is true for $n=3$. One applies the argument from the proof of Theorem \ref{thm:WirtDiscrHigh} to show that $\langle Mx, x \rangle$ has largest possible rank. See \cite{Izm14} for details. This argument uses the negative semidefiniteness of the quadratic forms of the links of vertices (Theorem \ref{thm:WirtDiscr}), thus the positive curvature condition is essential. Then one deforms the spherical cone metric on $\Sph^2$ to the nonsingular spherical metric, keeping all curvatures non-negative. This last step seems difficult to perform in higher dimensions. Ideally one would like to derive a ``discrete Weitzenboeck formula'' in analogy to the proof of the Lichnerowicz theorem.

\end{document}